\documentclass[leqno,12pt]{amsart} %leqno is the option to put formula numbers on the left side
\usepackage{amssymb}
\theoremstyle{plain} %text of this environment is typesetted in italics
\newtheorem{theorem}{\indent\sc Theorem}[section]
\newtheorem{lemma}[theorem]{\indent\sc Lemma}

\newtheorem{proposition}[theorem]{\indent\sc Proposition}

\theoremstyle{definition} %text of this environment is typesetted in roman letters

\newtheorem{define}[theorem]{\indent\sc Definition}
\newtheorem{remark}[theorem]{\indent\sc Remark}

\begin{document}
\date{} 

\title [Modified defect relations of the Gauss map and the total curvature]{Modified defect relations of the Gauss map and the total curvature of a complete minimal surface}

\author{Pham Hoang Ha}

\keywords{Minimal surface, Gauss map, Modified defect relation, Value distribution theory, Total curvature}

\subjclass[2010]{Primary 53A10; Secondary 53C42, 30D35, 32H30}
\address{
	Department of Mathematics,  \endgraf
	Hanoi National University of Education, \endgraf
	136, XuanThuy str., Hanoi, \endgraf
	Vietnam
}
\email{ha.ph@hnue.edu.vn}

\begin{abstract} 
In this article, we propose some conditions on the modified defect relations of the Gauss map of a complete minimal surface $M$ to show that $M$ has finite total curvature. 
\end{abstract}
\maketitle
%\tableofcontents
\section{Introduction}
In 1988, Fujimoto (\cite{Fu88}) proved Nirenberg's conjecture that if $M$ is a complete non-flat minimal surface in $\mathbb R^3,$ then its Gauss map can omit at most 4 points, and there are  a number of examples showing that the bound is sharp (see \cite[p.72-74]{O86}). He (\cite{Fu93}) also extended that result to the Gauss maps of complete minimal surfaces in $\mathbb R^m(m>3).$ For the case of a complete minimal surface with finite total curvature in $\mathbb R^3$, Osserman (\cite{O64}) proved its Gauss map can omit at most 3 points. We note that a complete minimal surface with finite total curvature to be called an algebraic minimal surface. Many results related to this topic were given (see \cite{Ru}, \cite{Fa}, \cite{JR}, \cite{KKM08}, \cite{HT15} and \cite{H17} for examples). Moreover, Mo and Osserman (\cite{MO}) showed an interesting improvement of Fujimoto's result by proving that a complete minimal surface in $\mathbb R^3$ whose Gauss map assumes five values only a finite number of times has finite total curvature.  After that, Mo (\cite{M}) extended that result to the complete minimal surface in $\mathbb R^m (m>3).$ Recently, the author, Phuong and Thoan (\cite{HPT16}) improved these results by giving some conditions on the ramifications of the Gauss map of a complete minimal surface $M$ in $\mathbb R^m (m \geq 3)$ to show that $M$ has finite total curvature.

On the other hand, in 1983, Fujimoto (\cite{Fu83}) introduced the non-intergrated defect relations for the Gauss map of a complete minimal surface which are similar to the defect relations given by R. Nevanlinna in his value distribution theory. After that, he also showed the modified defect relations for the Gauss maps of complete minimal surfaces to improve the previous results on the value distribution theory of the Gauss maps of  complete minimal surfaces relating to the omitted-properties, ramification properties (\cite{Fu89},\cite{Fu90}). Recently, the author and Trao (\cite{HT15}), the author (\cite{H17}) studied the non-integrated defect relations for the Gauss map of a complete minimal surface with finite total curvature in $\mathbb R^m(m\geq 3).$ These are the strict improvements of all previous results of Fujimoto on the modified defect relations for the Gauss map of a complete minimal surface with finite total curvature in $\mathbb R^m.$ 

A natural question is whether we may show a relation between of
the modified defect relations of the Gauss map and the total curvature of a complete minimal surface. In this article, we would like to give an affirmative answer for this question. More precisely, we introduce some conditions on the modified defect relations of the Gauss map of a complete minimal surface $M$ to show that $M$ has finite total curvature. 

The article is organized as follows: In Section \ref{s2}, we
recall some notations on the modified defect for a nonconstant holomorphic map of an open Riemann into $\mathbb P^k(\mathbb C).$ After that we introduce two main theorems of this article and using them to give some previous known results on the value-distribution-theoretic properties for the Gauss maps of complete minimal surfaces (Theorem \ref{T1}, \ref{T2}, \ref{T3} and \ref{T4}). In Section \ref{s3}, we give some lemmas which need for the proofs of the main theorems. Specially, we prove the lemma \ref{ML}, which is a generalization of the main lemma of Fujimoto in \cite{Fu90} by insteading the general position condition of hyperplanes by the subgeneral position one. In the last of this section we also show a relation between the classical defect in value distribution theory of meromorphic functions and modified defect. We will complete the main theorems in Section \ref{s4} and \ref{s5}. We present the proofs basing on the manners of the proofs of the main theorems in \cite{M} and \cite{HPT16}.

Finally, the author would like to thank Professors Nguyen Quang Dieu and Yu Kawakami for their valuable comments.
\section{Statements of the main results}\label{s2}
Let $M$ be an open Riemann surface and $f$ a nonconstant holomorphic map of $M$ into $\mathbb P^k(\mathbb C).$ Assume that $f$ has reduced representation $f=(f_0: \cdots : f_k).$ Set $||f|| = (|f_0|^2 +\cdots + |f_k|^2)^{1/2}$ and, for each a hyperplane $H : \overline{a}_{0}w_0+\cdots+ \overline{a}_{k}w_k = 0 $ in $\mathbb P^k(\mathbb C)$ with $|a_0|^2 +\cdots+ |a_k|^2 = 1,$ we define $f(H) := \overline{a}_0f_0  +\cdots+\overline{a}_kf_k.$
\begin{define} \label{d1}
	We define the $S-$defect of $H$ for $f$ by
	$$\delta^{S}_{f,M}(H):= 1 - \inf \{\eta \geq 0; \eta \text{ satisfies condition $(*)_S$}\}.$$
	Here, condition $(*)_S$ means that there exists a $ [-\infty, \infty)-$valued continuous subharmonic function $u\  (\not\equiv - \infty)$ on $M$ satisfying the following conditions:\\
	(C1) $e^{u} \leq ||f||^{\eta},$ \\
	(C2) for each $\xi \in f^{-1}(H) ,$ there exists the limit\\
	$$\lim_{z \rightarrow \xi}(u(z) - \min (\nu_{f(H)}(\xi), k)\log|z-\xi|) \in [-\infty, \infty) ,$$
	where $z$ is a holomorphic local coordinate around $\xi$ and $\nu_{f(H)}$ is the divisor of $f(H).$
\end{define}
\begin{remark}
	We always have that $\eta = 1$ satisfies condition $(*)_S$ with $u = \log |f(H)|.$
\end{remark}
\begin{define}
	\label{d2}
	We define the $H-$defect of $H$ for $f$ by
	$$\delta^H_{f,M}(H):= 1 - \inf \{\eta \geq 0; \eta \text{ satisfies condition $(*)_H$}\}.$$
	Here, condition $(*)_H$ means that there exists a $ [-\infty, \infty)-$valued continuous subharmonic function $u$ on $M$ which is harmonic on $M - f^{-1}(H)$ and satisfies the  conditions (C1) and (C2).
\end{define}
\begin{define} \label{d3}
	We define the $O-$defect of $H$ for $f$ by
	$$\delta^O_{f,M}(H):= 1 - \inf \{\ \dfrac{1}{n}; \ \text{ $f(H)$ has no zero of order less than $n$}\}.$$
\end{define}
\begin{remark}\label{rm1}
	We always have $0\leq \delta^O_{f,M}(H ) \leq \delta^H_{f,M}(H)\leq \delta^S_{f,M}(H)\leq 1.$
\end{remark}
\begin{define}
	One says that $f$ is ramified over a hyperplane $H$ in $\mathbb P^{k}(\mathbb C)$ {\it with multiplicity at least} $e$ if all the zeros of the function $f(H)$ have orders at least $e.$ If the image of $f$ omits $H,$ one will say that $f$ is {\it ramified over H with multiplicity }$\infty.$
\end{define}
\begin{remark} \label{rm2}
	If  $f$ is ramified over a hyperplane $H$ in $\mathbb P^{k}(\mathbb C)$ with multiplicity at least $n,$ then $\delta^S_{f,M}(H) \geq \delta^H_{f,M}(H) \geq \delta^O_{f,M}(H ) \geq 1 - \dfrac{1}{n}.$ In particular, if $f^{-1}(H) = \emptyset,$ then $\delta^O_{f,M}(H ) = 1.$ 
\end{remark}

\indent Let  $x=(x_0, \cdots, x_{m-1}): M \to \mathbb R^m$ be  a  (smooth,  oriented)  minimal  surface  immersed  in $\mathbb R^m$. Then  $M$ has the structure of a Riemann surface and any local isothermal coordinate $(\xi_1, \xi_2)$ of $M$ gives a local holomorphic coordinate   $z =\xi_1+\sqrt{-1}\xi_2.$ The (generalized) Gauss  map  of  $x$  is  defined  to  be
\begin{align*}
	g: M \to Q_{m-2}(\mathbb C) \subset \mathbb P^{m-1}(\mathbb C), g(z) =(\dfrac{\partial x_0}{\partial z}:\cdots:\dfrac{\partial x_{m-1}}{\partial z}),
\end{align*}          
where 
\begin{align*}
	Q_{m-2}(\mathbb C)=\{(w_0: \cdots : w_{m-1})| w_0^2 + \cdots + w_{m-1}^2 = 0\} \subset \mathbb P^{m-1}(\mathbb C).
\end{align*}   
By the  assumption  of  minimality  of  $M, g$  is  a holomorphic  map  of  $M$  into
$Q_{m-2}(\mathbb C).$ 

In this article, we would like to study the relations between $H-$ defect relations for the Gauss maps with the total curvature of minimal surfaces in  $\mathbb R^m.$  In particular, we will prove the followings.\\

{\bf Main theorem 1.}\ {\it
		Let $M$ be a complete minimal surface in $\mathbb R^m$ and $K$ be a compact subset in $M.$ Set $A:= M \setminus K.$ Let $H_1,...,H_q$ be hyperplanes in $ \mathbb P^{m-1}(\mathbb C)$ located in $N$-subgeneral position $( q > 2N - k + 1, N \geq m-1).$ Assume that the generalized Gauss map $g$ of $M$ is $k-$non-degenerate (that is $g(M)$ is contained in a $k-$dimensional linear subspace in $ \mathbb P^{m-1}(\mathbb C)$, but none of lower dimension), $1\leq k \leq m-1,$ and 
	\begin{equation*}
		\sum_{j=1}^q\delta^H_{g,A}(H_j) > (k+1)(N-\dfrac{k}{2})+(N+1)
		\end{equation*} 
	then $M$ has finite total curvature. }

When $m=3,$ we can identify $\mathbb Q_1(\mathbb C)$ with $\mathbb P^1(\mathbb C).$ So we can get a better result as the following:\\

{\bf Main theorem 2.}\ {\it
	Let $M$ be a complete minimal surface in $\mathbb R^3$ and $K$ be a compact subset in $M.$ Set $A:= M \setminus K.$ Let $a_1,...,a_q$ be $q$ distinct points in $ \mathbb P^{1}(\mathbb C)$. Assume that the  Gauss map $g$ of $M$ satisfies 
	\begin{equation*}
		\sum_{j=1}^q\delta^H_{g,A}(a_j) > 4
	\end{equation*} 
	then $M$ has finite total curvature. }

As the corollaries of the main theorems we can get the following theorems: 
\begin{theorem} [{\cite[Theorem 1]{HPT16}}] \label{T1}
	Let $M$ be a complete minimal surface in $\mathbb R^m$ and $K$ be a compact subset in $M.$ Assume that the generalized Gauss map $g$ of $M$ is $k-$non-degenerate, $1\leq k \leq m-1$. If there are $q$ hyperplanes  $\{H_j\}_{j=1}^q$ located in $N$-subgeneral position in $ \mathbb P^{m-1}(\mathbb C), (N \geq m-1)$ such that $g$ is ramified over $H_j$ with multiplicity at least $m_j$ on $M\setminus K$ for each $j$ and  
	\begin{equation*}
		\sum_{j=1}^q(1 - \frac{k}{m_j})> (k+1)(N-\dfrac{k}{2})+(N+1),
	\end{equation*} 
	then $M$ has finite total curvature.\\
	In particular, if $\{H_j\}_{j=1}^q$ are in general position in $ \mathbb P^{m-1}(\mathbb C)$ and 
	\begin{equation*}
		\sum_{j=1}^q(1 - \frac{m-1}{m_j})> \dfrac{m(m+1)}{2}, 
	\end{equation*} 
	then $M$ must have finite total curvature. 
\end{theorem}

\begin{theorem}[{\cite[Theorem 2]{HPT16}}] \label{T2}
	Let $M$ be a complete minimal surface in $\mathbb R^3$  and $q$ distinct points $a_1, ... , a_{q}$ in $\mathbb P^1(\mathbb C).$ Suppose that the Gauss map $g$ of $M$ is ramifed over $a_j$ with multiplicity at least $m_j$ for each $j=1,\cdots, q$ outside a compact subset $K$ of  $M$. Then $M$ has finite total curvature if 
	\begin{equation*}
	 \sum\limits_{j=1}^{q}\big(1-\frac{1}{m_j}\big) > 4.
	\end{equation*}  
\end{theorem}
Now, applying the main theorems we can prove the following results:
\begin{theorem} [{\cite[Theorem 2.8]{Fu90}}] \label{T3}
	For any complete minimal surface $M$ immersed in $\mathbb R^m.$ Assume that the generalized Gauss map $g$ of $M$ is $k-$non-degenerate, $1\leq k \leq m-1$. If there are $q$ hyperplanes  $\{H_j\}_{j=1}^q$ located in $N$-subgeneral position in $ \mathbb P^{m-1}(\mathbb C).$ Then
	\begin{equation*}\label{4}
		\sum_{j=1}^q\delta^H_{g,M}(H_j) \leq (k+1)(N-\dfrac{k}{2})+(N+1)
	\end{equation*}
	In particular, let $\{H_j\}_{j=1}^q$ be $q$ hyperplanes in general position in $\mathbb P^{m-1}(\mathbb C).$ If  
	$$\sum_{j=1}^q\delta^H_{g,M}(H_j) > \dfrac{m(m+1)}{2}$$
	then $M$ is flat, or equivalently, $g$ is constant.
\end{theorem}
\begin{proof}
	Assume that 
		\begin{equation*}
		\sum_{j=1}^q\delta^H_{g,M}(H_j) > (k+1)(N-\dfrac{k}{2})+(N+1). 
	\end{equation*}
Using Main theorem 1 for the case $K = \emptyset$, we get that $M$ has finite total curvature. Now thanks to the main theorem in \cite{H17}, we have
$$\sum_{j=1}^q\delta^H_{g,M}(H_j) \leq \sum_{j=1}^q\delta^S_{g,M}(H_j) \leq (k+1)(N-\dfrac{k}{2})+(N+1).$$ 
This is a contradiction. The theorem \ref{T3} is proved.
	\end{proof}
\begin{theorem} [{\cite[Theorem I]{Fu89}}] \label{T4}
	For any non-flat complete minimal surface $M$ immersed in $\mathbb R^3$ with its Gauss map $g.$ Let $\{a_j\}_{j=1}^q$ be $q$ distinct points in $ \mathbb P^{1}(\mathbb C).$ Then
	\begin{equation*}\label{4}
		\sum_{j=1}^q\delta^H_{g,M}(a_j) \leq 4.
	\end{equation*}
\end{theorem}
\begin{proof}
	Assume that 
	\begin{equation*}
		\sum_{j=1}^q\delta^H_{g,M}(a_j) > 4. 
	\end{equation*}
	Since Main theorem 2 with $K = \emptyset$, we get that $M$ has finite total curvature. Now we use the theorem 1 in \cite{HT15}, we get
	$$\sum_{j=1}^q\delta^H_{g,M}(a_j) \leq \sum_{j=1}^q\delta^S_{g,M}(a_j) \leq 4.$$ 
	This is a contradiction. The theorem \ref{T4} is proved.
\end{proof}
\section{Auxiliary  lemmas}\label{s3}
Let  $M$  be an open Riemann surface and $ds^2$ a pseudo-metric on $M$, namely, a metric on $M$ with isolated singularities which is locally written as $ds^2=\lambda ^2\left|dz\right|^2$ in terms of a nonnegative real-value function $\lambda$  with mild singularities and a holomorphic local coordinate $z$. We define the divisor of  $ds^2$ by $\nu_{ds}:=\nu_{\lambda}$ for each local expression $ds^2=\lambda ^2 \left|dz\right|^2$, which is globally well-defined on $M$. We say that $ds^2$ is a continuous pseudo-metric if $\nu_{ds}\geq 0$ everywhere.
\begin{define} [see \cite{Fu93}]
	We define the Ricci form of $ds^2$ by 
	$$\mbox{Ric}_{ds^2}:= -dd^c\log\lambda ^2$$
	for each local expression $ds^2=\lambda ^2 \left|dz\right|^2.$ 
\end{define}
\noindent In some cases, a $(1,1)-$form $\Omega$ on $M$ is regarded as a current on $M$ by defining $\Omega(\varphi):= \int_{M}\varphi\Omega$ for each $\varphi \in \mathcal{D},$ where $\mathcal{D}$ denotes the space of all $C^{\infty}$ differentiable functions on $M$ with compact supports.
\begin{define} [see \cite{Fu93}] 
	We say that a continuous pseudo-metric $ds^2$ has strictly negative curvature on $M$ if there is a positive constant $C$ such that 
	$$- \mbox{Ric}_{ds^2}\geq C\cdot\Omega_{ds^2},$$
	where $ \Omega_{ds^2}$ denotes the area form for $ds^2$, namely,
	$$\Omega_{ds^2}:= \lambda ^2 (\sqrt{-1}/2)dz \wedge d\bar{z}$$
	for each local expression $ds^2=\lambda ^2\left| dz\right|^2$.
\end{define}
\indent Let $f$ be a linearly non-degenerate holomorphic map of $M$ into $\mathbb P^k(\mathbb C).$ 
Take a reduced representation $f = (f_0: \cdots : f_k)$. Then $F := (f_0, \cdots, f_k): M \rightarrow \mathbb C^{k+1} \setminus \{0\}$ is a holomorphic map with $\mathbb P(F) = f.$ 
Consider the holomorphic map
\begin{equation*}
F_p=(F_p)_{z}:=F^{(0)}\wedge F^{(1)}\wedge\cdots\wedge F^{(p)}: M \longrightarrow \wedge^{p+1}\mathbb{C}^{k+1}
\end{equation*}
for $0\le p \le k,$  where  $F^{(0)} := F= (f_0,\cdots,f_k)$ and $F^{(l)}=(F^{(l)})_{z}:=(f_0^{(l)},\cdots,f_k^{(l)})$ for each $l=0,1,\cdots,k$, and where the $l$-th derivatives 
$f_i^{(l)}=(f_i^{(l)})_{z}$, $i=0,...,k$, are taken with respect to $z$.
(Here and for the rest of this paper the index $|_{z}$ means that the corresponding term
is defined by using differentiation with respect to the variable $z$, and in order to keep notations simple, we usually drop this index if no confusion is possible).
The norm of $F_p$ is given by
\begin{equation*}
\left| F_p\right|:= \bigg(\sum_{0\leq i_0<\cdots<i_p\leq k}\left|W(f_{i_0},\cdots,f_{i_p})\right|^2\bigg)^\frac{1}{2},
\end{equation*}
where $W(f_{i_0},\cdots,f_{i_p}) = W_z(f_{i_0},\cdots,f_{i_p})$ denotes the Wronskian of $f_{i_0},\cdots,f_{i_p}$ with respect to $z$.
\begin{proposition}[{\cite[Proposition 2.1.6]{Fu93}}]\label{W}.\\
	For two holomorphic local coordinates $z$ and $\xi$ and a holomorphic function 
	$h : M \rightarrow \mathbb{C}$, the following holds :\\
	a) $  W_{\xi}(f_0,\cdots, f_p)= W_z(f_0,\cdots, f_p) \cdot (\frac{dz}{d\xi})^{p(p+1)/2}$.\\
	b) $W_{z}(hf_0,\cdots, hf_p)= W_z(f_0,\cdots, f_p) \cdot (h)^{p+1}. $ 
\end{proposition}
\begin{proposition}[{\cite[Proposition 2.1.7]{Fu93}}]\label{W1}.\\
	For holomorphic functions $f_0, \cdots , f_p : M \rightarrow \mathbb{C}$ the following conditions are equivalent:\\ 
	(i)\ $ f_0, \cdots, f_p$ are linearly dependent over $\mathbb C.$\\
	(ii)\ $W_z(f_0,\cdots,f_p) \equiv 0$ for some (or all) holomorphic local coordinate $z.$
\end{proposition}
We now take a hyperplane $H$ in $\mathbb P^k(\mathbb C)$ given by
\begin{equation*}
H:\overline{c}_0\omega_0+\cdots+\overline{c}_k\omega_k=0\,,
\end{equation*}
with $\sum_{i=0}^k|c_i|^2 = 1.$
We set 
\begin{equation*}
F_0(H) :=F(H):=\overline{c}_0f_0+\cdots+\overline{c}_kf_k
\end{equation*}
and 
\begin{equation*}
\left| F_p(H)\right|=\left| (F_p)_{z}(H)\right|:= \bigg(\sum_{0\leq i_1< \cdots<i_p\leq k}\left|\sum_{l\not= i_1,...,i_p}\overline{c}_lW(f_{l},f_{i_1}, \cdots,f_{i_p})\right|^2\bigg)^\frac{1}{2},
\end{equation*}
for $1\leq p \leq k.$ We note that by using Proposition \ref{W}, $\left| (F_p)_{z}(H)\right|$ is multiplied by a factor $|\frac{dz}{d\xi}|^{p(p+1)/2}$ if we choose another holomorphic local coordinate $\xi$, and it is multiplied by $|h|^{p+1}$ if we choose another reduced representation $f=(hf_0:\cdots:hf_k)$ with a nowhere zero holomorphic function $h.$
Finally, for $0 \leq p \leq k$, set the $p$-th contact function of $f$ for $H$ to be $\phi_p(H):=\dfrac{|F_p(H)|^2}{|F_p|^2}=\dfrac{|(F_p)_{z}(H)|^2}{|(F_p)_{z}|^2}.$\\

We next consider $q$ hyperplanes $H_1,\cdots,H_q$ in $\mathbb{P}^{k}(\mathbb{C})$ given by
$$H_j:\left\langle \omega ,A_j\right\rangle \equiv \overline{c}_{j0}\omega_0+\cdots+\overline{c}_{jk}\omega_k \quad(1\leq j\leq q)$$
where $A_j:=(c_{j0},\cdots,c_{jk})$ with $\sum_{i=0}^k|c_{ji}|^2 = 1.$\\
Assume now $N \geq k$ and $q \geq N+1$.
For $R\subseteq Q:=\left\{1,2,\cdots,q\right\},$ denote by $d(R)$ the dimension of the vector subspace of $\mathbb C^{k+1}$ generated by $\left\{A_j;j\in R \right\}$. The hyperplanes $H_1,\cdots,H_q$ are said to be in $N$-subgeneral position if $d(R)=k+1$ for all $R\subseteq Q$ with $\sharp (R)\geq N+1,$ where $\sharp (A)$ means the number of elements of a set $A.$ In the particular case $N=k$, these are said to be in general position.
\begin{theorem}[{\cite[Theorem 2.4.11]{Fu93}}] \label{N1}
	\emph {\it For given hyperplanes $H_1,\cdots,H_q $ $( q > 2N - k + 1)$ in $\mathbb{P}^k(\mathbb{C})$ located in $N$-subgeneral position, there are some rational numbers $\omega(1),\cdots,\omega(q)$ and $\theta$ satisfying the following conditions:\\
		\indent (i) $0<\omega(j)\leq \theta\leq 1  \quad(1\leq j\leq q),$\\
		\indent (ii) $ \sum^{q}_{j=1}\omega(j)=k+1+\theta(q-2N+k-1),$\\
		\indent (iii) $\frac{k+1}{2N-k+1}\leq \theta \leq \frac{k+1}{N+1},$\\
		\indent (iv) If $R\subset Q$ and $0< \sharp(R)\leq n+1,$ then $\sum_{j \in R} \omega(j)\leq d(R).$}
\end{theorem}
\noindent  Constants $\omega(j)\ (1 \leq  j  \leq q)$ and $\theta$ with the properties of Theorem~\ref{N1} are called Nochka weights and a Nochka constant for $H_1, \cdots, H_q$ respectively.

We need the three following results of Fujimoto combining the previously introduced concept of contact functions with Nochka weights:
\begin{theorem}[{\cite[Theorem 2.5.3]{Fu93}}] \label{PL1}
	Let $H_1, \cdots, H_q$ be hyperplanes in $\mathbb P^k(\mathbb C)$ located in $N-$subgeneral position and let $\omega(j)$ $(1\leq j \leq q)$ and $\theta$ be Nochka weights and a Nochka constant for these hyperplanes. For every $\epsilon > 0$ there exist some positive numbers $\delta(>1)$ and $C,$ depending only on $\epsilon$ and $H_j\,$, $1\leq j \leq q,$ such that
	\begin{align*}\label{eq:T1}
	&dd^c\log \dfrac{\Pi_{p=0}^{k-1}|F_p|^{2\epsilon}}{\Pi_{1\leq j\leq q, 0\leq p\leq k-1}\log^{2\omega(j)}(\delta/\phi_p(H_j))}\nonumber \\
	&\geq C\bigg( \dfrac{|F_0|^{2\theta(q-2N+k-1)}|F_k|^2}{\Pi_{j=1}^q(|F(H_j)|^2\Pi_{p=0}^{k-1}\log^2(\delta/\phi_p(H_j)))^{\omega(j)}}\bigg)^{\frac{2}{k(k+1)}}dd^c|z|^2.
	\end{align*}
\end{theorem}
\begin{proposition}[{\cite[Proposition 2.5.7]{Fu93}}]\label{P} \
	Set $\sigma_p=p(p+1)/2$ for $0 \leq p \leq k$ and $\tau_k = \sum_{p=0}^k\sigma_p.$ Then,
	\begin{align*}
	dd^c\log(|F_0|^2|F_1|^2\cdots|F_{k-1}|^2)\geq \dfrac{\tau_k}{\sigma_k}\bigg(\dfrac{|F_0|^2|F_1|^2\cdots|F_{k}|^2}{|F_0|^{2\sigma_{k+1}}}\bigg)^{1/\tau_k}dd^c|z|^2.
	\end{align*}
\end{proposition}
\begin{proposition} [{\cite[Lemma 3.2.13]{Fu93}}]
	\label{P7}
	Let $f$ be a non-degenerate holomorphic map of a domain in  $\mathbb{C}$ into $\mathbb{P}^k(\mathbb{C})$ with reduced representation $f=(f_0:\cdots :f_k)$ and let $H_1,\cdots,H_q$ be hyperplanes located in $N$-subgeneral position $( q > 2N - k + 1)$ with Nochka weights $\omega(1),\cdots,\omega(q)$ respectively. Then,
	$$\nu_\phi + \sum^{q}_{j=1}\omega(j)\cdot\min (\nu_{(f, H_j)}, k ) \geq 0,$$
	where $\phi = \dfrac{|F_k|}{\Pi_{j=1}^q \mid F(H_j)\mid^{\omega (j)}}$ and $\nu_{\phi}$ is the divisor of $\phi.$
\end{proposition}
\begin{lemma}[{Generalized Schwarz's Lemma \cite{Ah}}] \label{L3}
	Let $v$ be a non-negative real-valued continuous subharmonic function on $\Delta_R(=\{ z \in \mathbb C \mid |z| <R\}).$ If $v$ satisfies the inequality  $\Delta\log v \geq v^2$ in the sense of distribution, then
	$$v(z) \leq \dfrac{2R}{R^2 - |z|^2}.$$
\end{lemma}
\begin{lemma}\label{ML}
	\ Let $f= ( f_0: \cdots : f_k ) : \Delta_R  \rightarrow \mathbb P^k (\mathbb{C})$ be a non-degenerate holomorphic map, $ H_1, ... , H_q$ be hyperplanes in $ \mathbb P^k (\mathbb{C}) $ in $N-$subgeneral position ($N \geq k$ and $q > 2N-k+1$), and $\omega(j)(1 \leq j \leq q)$ be their Nochka weights. Assume that there are positive real numbers $\eta_j (1 \leq j \leq q)$ and $[-\infty, \infty)-$valued continuous subharmonic functions $u_j$ sastifying conditions (C1), (C2). Then for an arbitrarily given $\epsilon$ satisfying $\gamma=\sum_{j=1}^q\omega(j)(1-\eta_j)-(k+1) >\epsilon\sigma_{k+1} >0$, there exists a positive constant  $C,$ depending only on $\epsilon, H_j,\eta_j,\omega(j) (1\leq j \leq q),$ such that
	\begin{equation*}
	\dfrac{|F|^{\gamma -\epsilon\sigma_{k+1}}e^{\sum_{j=1}^q\omega(j)u_j}.|F_k|^{1+\epsilon}.\Pi_{j=1}^q\Pi_{p=0}^{k-1}|F_p(H_j)|^{\epsilon/q}}
	{\Pi_{j=1}^q|F(H_j)|^{\omega(j)}} \leqslant C\bigg(\dfrac{2R}{R^2 -|z|^2}\bigg)^{\sigma_k +\epsilon\tau_k},
	\end{equation*}
	where $\sigma_p=p(p+1)/2$ for $0 \leq p \leq k$ and $\tau_k = \sum_{p=0}^k\sigma_p.$
\end{lemma}
\begin{proof}
We set
\begin{equation*}
\lambda_z := \bigg( \dfrac{|F_z|^{\gamma -\epsilon\sigma_{k+1}}e^{\sum_{j=1}^q\omega(j)u_j}.|(F_k)_z|.\prod_{p=0}^{k}|(F_p)_z|^{\epsilon}}{\prod_{j=1}^{q}(|F(H_j)|\Pi_{p=0}^{k-1}\log(\delta/\phi_p(H_j)))^{\omega (j)}}\bigg)^{\frac{1}{\sigma_k +\epsilon\tau_k}},
\end{equation*}
and define the pseudo-metric $d\tau_z^2 := \lambda_z^2|dz|^2.$ Using Proposition \ref{W}, we can see that 
\begin{align*}
d\tau_{\xi} &:= \bigg( \dfrac{|F_{\xi}|^{\gamma -\epsilon\sigma_{k+1}}e^{\sum_{j=1}^q\omega(j)u_j}.|(F_k)_{\xi}|.\prod_{p=0}^{k}|(F_p)_{\xi}|^{\epsilon}}{\prod_{j=1}^{q}(|F(H_j)|\Pi_{p=0}^{k-1}\log(\delta/\phi_p(H_j)))^{\omega (j)}}\bigg)^{\frac{1}{\sigma_k +\epsilon\tau_k}}|d\xi|\\
&= \bigg( \dfrac{|F_z|^{\gamma -\epsilon\sigma_{k+1}}e^{\sum_{j=1}^q\omega(j)u_j}.|(F_k)_{z}||\frac{dz}{d\xi}|^{\sigma_k}.\prod_{p=0}^{k}|(F_p)_{z}|^{\epsilon}.|\frac{dz}{d\xi}|^{\sum_{p=0}^{k}\epsilon\frac{p(p+1)}{2}}}{\prod_{j=1}^{q}(|F(H_j)|\Pi_{p=0}^{k-1}\log(\delta/\phi_p(H_j)))^{\omega (j)}}\bigg)^{\frac{1}{\sigma_k +\epsilon\tau_k}}|\dfrac{d\xi}{dz}|.|dz|\\
&= \bigg( \dfrac{|F_z|^{\gamma -\epsilon\sigma_{k+1}}e^{\sum_{j=1}^q\omega(j)u_j}.|(F_k)_{z}|.\prod_{p=0}^{k}|(F_p)_{z}|^{\epsilon}.|\frac{dz}{d\xi}|^{\sigma_k +\epsilon\tau_k}}{\prod_{j=1}^{q}(|F(H_j)|\Pi_{p=0}^{k-1}\log(\delta/\phi_p(H_j)))^{\omega (j)}}\bigg)^{\frac{1}{\sigma_k +\epsilon\tau_k}}|\dfrac{d\xi}{dz}|.|dz|\\
&=d\tau_{z}.
\end{align*}
Thus $d\tau_z^2$ is independent of the choice of the local coordinate $z.$ We will denote $d\tau_z^2$ by $d\tau^2$ for convenience. So $d\tau^2$ is well-defined on $M$ and 
\begin{equation*}
d\tau^2 =\bigg( \dfrac{|F|^{\gamma -\epsilon\sigma_{k+1}}e^{\sum_{j=1}^q\omega(j)u_j}.|F_k|.\prod_{p=0}^{k}|F_p|^{\epsilon}}{\prod_{j=1}^{q}(|F(H_j)|\Pi_{p=0}^{k-1}\log(\delta/\phi_p(H_j)))^{\omega (j)}}\bigg)^{\frac{2}{\sigma_k +\epsilon\tau_k}}|dz|^2:= \lambda^2|dz|^2.
\end{equation*}
\indent We will show that $d\tau^2$ is continuous and has strictly negative curvature on $M$. \\
Indeed, it is easy to see that $d\tau$ is continuous at every point $z_0$ with $\Pi_{j=1}^qF(H_j)(z_0) \not=0.$ Now we take a point $z_0 $ such that $\Pi_{j=1}^qF(H_j)(z_0) =0.$ Hence, it follows from (C2) that we get 
\begin{align*}
\lim_{z \to z_0} e^{u_j(z)}|z-z_0|^{-\min(\nu_{F(H_j)}(z_0), k)} =\lim_{z \to z_0} e^{u_j(z) - \min(\nu_{F(H_j)}(z_0), k)\log|z-z_0|} < \infty, \forall \  1\leq j \leq q.
\end{align*}
Thus, combining this with Proposition \ref{P7}, we obtain
\begin{align*}
\nu_{d\tau}(z_0)& \geq \dfrac{1}{\sigma_k + \epsilon\tau_k}\bigg( \nu_{F_k}(z_0)-\sum_{j=1}^q\omega(j)\nu_{F(H_j)}(z_0)+\sum_{j=1}^q\omega(j)\min\{\nu_{F(H_j)}(z_0), k \}\\
& +\sum_{j=1}^q\omega(j)\nu_{e^{u_j(z)}|z-z_0|^{-\min(\nu_{F(H_j)}(z_0), k)}}(z_0) \bigg) \\ 
&\geq 0.
\end{align*}
This implies that $d\tau$ is continuous on $M.$\\
On the other hand, by using Proposition \ref{P}, Theorem \ref{PL1} and noting that $dd^c\log|F_k| =0, dd^c\log e^{u_j} \geq 0 (1\leq j\leq q)$, we have
\begin{align*}
dd^c\log\lambda&\geq \dfrac{\gamma - \epsilon\sigma_{k+1}}{\sigma_{k}+\epsilon\tau_k}dd^c\log|F|+\dfrac{\epsilon}{4(\sigma_{k}+\epsilon\tau_k)}dd^c\log(|F_0|^2\cdots|F_{k-1}|^2)\\ 
& + \dfrac{1}{2(\sigma_k +\epsilon\tau_k)}dd^c\log\dfrac{\prod_{p=0}^{k-1}|F_p|^{2(\frac{\epsilon}{2})}}{\prod_{j=1}^{q}\Pi_{p=0}^{k-1}\log^{2\omega (j)}(\delta/\phi_p(H_j))}\\
&\geq \dfrac{\epsilon}{4(\sigma_{k}+\epsilon\tau_k)}\dfrac{\tau_k}{\sigma_k}\bigg(\dfrac{|F_0|^2|F_1|^2\cdots|F_{k}|^2}{|F_0|^{2\sigma_{k+1}}}\bigg)^{1/\tau_k}dd^c|z|^2\\ 
& + C_0\bigg( \dfrac{|F_0|^{2\theta(q-2N+k-1)}|F_k|^2}{\Pi_{j=1}^q(|F(H_j)|^2\Pi_{p=0}^{k-1}\log^2(\delta/\phi_p(H_j)))^{\omega(j)}}\bigg)^{\frac{2}{k(k+1)}}dd^c|z|^2\\
&\geq \min \{\dfrac{1}{4\sigma_k(\sigma_{k}+\epsilon\tau_k)}, \dfrac{C_0}{\sigma_k} \}\bigg(\epsilon\tau_k\bigg(\dfrac{|F_0|^2|F_1|^2\cdots|F_{k}|^2}{|F_0|^{2\sigma_{k+1}}}\bigg)^{1/\tau_k}\\
&+ \sigma_k\bigg( \dfrac{|F_0|^{2\theta(q-2N+k-1)}|F_k|^2}{\Pi_{j=1}^q(|F(H_j)|^2\Pi_{p=0}^{k-1}\log^2(\delta/\phi_p(H_j)))^{\omega(j)}}\bigg)^{\frac{1}{\sigma_k}}\bigg) dd^c|z|^2
\end{align*}
where $C_0$ is the positive constant. So,  by using the basic inequality 
$$ \alpha A + \beta B \geq (\alpha + \beta) A^{\frac{\alpha}{\alpha+ \beta}}B^{\frac{\beta}{\alpha + \beta}} \text{ for all }\alpha, \beta, A, B > 0  ,$$ we can find a positive constant $C_1$ satisfing the following\\
\begin{align*}
dd^c\log\lambda &\geq C_1\bigg( \dfrac{|F|^{\theta (q-2N+k-1) -\epsilon\sigma_{k+1}}.|F_k|.\prod_{p=0}^{k}|F_p|^{\epsilon}}{\prod_{j=1}^{q}(|F(H_j)|\cdot \Pi_{p=0}^{k-1}\log(\delta/\phi_p(H_j)))^{\omega (j)}}\bigg)^{\frac{2}{\sigma_k +\epsilon\tau_k}}dd^c|z|^2\\
&= C_1\bigg( \dfrac{|F|^{\sum_{j=1}^q \omega (j) - k-1 -\epsilon\sigma_{k+1}}.|F_k|.\prod_{p=0}^{k}|F_p|^{\epsilon}}{\prod_{j=1}^{q}(|F(H_j)|\cdot \Pi_{p=0}^{k-1}\log(\delta/\phi_p(H_j)))^{\omega (j)}}\bigg)^{\frac{2}{\sigma_k +\epsilon\tau_k}}dd^c|z|^2 \  \text{ (by Theorem \ref{N1}) }\\
&= C_1\bigg( \dfrac{|F|^{\gamma -\epsilon\sigma_{k+1}}.e^{\sum_{j=1}^q\omega(j)u_j}.|F_k|.\prod_{p=0}^{k}|F_p|^{\epsilon}.\prod_{j=1}^{q}\bigg(\dfrac{|F|^{\eta_j}}{e^{u_j}}\bigg)^{\omega (j)}}{\prod_{j=1}^{q}(|F(H_j)|\cdot \Pi_{p=0}^{k-1}\log(\delta/\phi_p(H_j)))^{\omega (j)}}\bigg)^{\frac{2}{\sigma_k +\epsilon\tau_k}}dd^c|z|^2.
\end{align*}
On the other hand, using (C1) we have
\begin{equation*}
\bigg(\dfrac{|F|^{\eta_j}}{e^{u_j}}\bigg)^{\omega (j)} \geq 1 \text{ for all } j = 1, ..., q.
\end{equation*}
Thus we get
\begin{align*}
dd^c\log\lambda &\geq C_1\bigg( \dfrac{|F|^{\gamma -\epsilon\sigma_{k+1}}.e^{\sum_{j=1}^q\omega(j)u_j}.|F_k|.\prod_{p=0}^{k}|F_p|^{\epsilon}}{\prod_{j=1}^{q}(|F(H_j)|\cdot \Pi_{p=0}^{k-1}\log(\delta/\phi_p(H_j)))^{\omega (j)}}\bigg)^{\frac{2}{\sigma_k +\epsilon\tau_k}}dd^c|z|^2\\
&=C_1\lambda^2dd^c|z|^2.
\end{align*}
This concludes the proof that $d\tau^2$ has strictly negative curvature on $M.$\\
Combining with Lemma \ref{L3} we now obtain
\begin{equation*}
\bigg( \dfrac{|F|^{\gamma -\epsilon\sigma_{k+1}}e^{\sum_{j=1}^q\omega(j)u_j}.|F_k|.\prod_{p=0}^{k}|F_p|^{\epsilon}}{\prod_{j=1}^{q}(|F(H_j)|\Pi_{p=0}^{k-1}\log(\delta/\phi_p(H_j)))^{\omega (j)}}\bigg)^{\frac{1}{\sigma_k +\epsilon\tau_k}} \leq C_2\dfrac{2R}{R^2 - |z|^2}.
\end{equation*}
Then
\begin{equation*}
	 	\bigg(\dfrac{|F|^{\gamma -\epsilon\sigma_{k+1}}e^{\sum_{j=1}^q\omega(j)u_j}.|F_k|^{1+\epsilon}.\Pi_{j=1}^q\Pi_{p=0}^{k-1}|F_p(H_j)|^{\epsilon/q}}
	{\Pi_{j=1}^q|F(H_j)|^{\omega(j)}\Pi_{j=1}^{q}\Pi_{p=0}^{k-1}(\phi_p(H_j))^{\frac{\epsilon}{2q}}(\log(\delta/\phi_p(H_j)))^{\omega (j)}}\bigg)^{\frac{1}{\sigma_k +\epsilon\tau_k}}\leq C_2\dfrac{2R}{R^2 - |z|^2}.
\end{equation*}
Since the function $x^{\frac{\epsilon}{q}}\log^{\omega}(\dfrac{\delta}{x^2}) ( \omega >0, 0 < x \leq 1)$ is bounded, we obtain that 
\begin{equation*}
	\bigg(\dfrac{|F|^{\gamma -\epsilon\sigma_{k+1}}e^{\sum_{j=1}^q\omega(j)u_j}.|F_k|^{1+\epsilon}.\Pi_{j=1}^q\Pi_{p=0}^{k-1}|F_p(H_j)|^{\epsilon/q}}
	{\Pi_{j=1}^q|F(H_j)|^{\omega(j)}}\bigg)^{\frac{1}{\sigma_k +\epsilon\tau_k}}\leq C_3 \dfrac{2R}{R^2 - |z|^2}
\end{equation*}
for some positive real number $C_3$. This proves Lemma \ref{ML}.
\end{proof}
\begin{remark}
	We remark that Lemma \ref{ML} is a generalization of the main lemma in \cite{Fu90}, in which the general position condition of hyperplanes is replaced by the subgeneral position one. 
	\end{remark}
In particular, we introduce the following version for the case one dimension.
\begin{lemma}[{\cite[Main Lemma]{Fu89}}]\label{L6}
	Let $f= ( f_0: \cdots : f_k ) : \Delta_R \rightarrow \mathbb P^1 (\mathbb{C})$ be a non-degenerate holomorphic map, $ a_1, ... , a_q$ be distinct points in $ \mathbb P^1 (\mathbb{C}) .$ Assume that there are positive real number $\eta_j (1 \leq j \leq q)$ and $[-\infty, \infty)-$valued continuous subharmonic functions $u_j$ sastifying conditions (C1), (C2) and $\sum_{j=1}^q(1-\eta_j)-2 >0.$ There exists a positive constant $C$ such that
	$$\dfrac{||f||^{q-2 -\sum_{j=1}^q\eta_j -  q\delta}e^{\sum_{j=1}^qu_j}|W(f_0,f_1)|}{\Pi_{j=1}^q|F(a_j)|^{1-\delta}}\leq C\dfrac{2R}{R^2 - |z|^2}\, .$$
\end{lemma}
We also need the following result on completeness of open Riemann surfaces with conformally flat metrics due to Fujimoto :
\begin{lemma}[{\cite[Lemma 1.6.7]{Fu93}}] \label{L5}
Let $d\sigma^2$ be a conformal flat metric on an open Riemann surface $M$. Then for every point $p \in M$, there is a holomorphic and locally biholomorphic map $\Phi$ of
a disk (possibly with radius $\infty$)  $\Delta_{R_0} := \{w : |w|<R_0 \}$ $(0<R_0 \leq \infty )$ onto an open neighborhood of $p$ with $\Phi (0) = p$ such that $\Phi$ is a local isometry, namely the pull-back 
$\Phi^*(d\sigma^2)$ is equal to the standard (flat) metric on $\Delta_{R_0}$, and for some point $a_0$ with $|a_0|=1$, the $\Phi$-image of the curve 
$$L_{a_0} : w:= a_0 \cdot s \; (0 \leq s < R_0)$$
is divergent in $M$ (i.e. for any compact set $K \subset M$, there exists an $s_0<R_0$
such that the $\Phi$-image of the curve $L_{a_0} : w:= a_0 \cdot s \; (s_0 \leq s < R_0)$
does not intersect $K$).
\end{lemma}
Finally, we would like to show a relation between the classical defect in value distribution theory of meromorphic functions and modified defect which need for our proof of the main theorems.\\
Consider the case $M=\Delta_{s,\infty}:=\{ z \in \mathbb C \mid  s\leq |z| <\infty\}$ where $s$ is a positive real number. The order function of $f,$ the counting function for $H$ and the classical Nevanlinna defect (truncated by $n$) are defined respectively by
	\begin{align*}
		T_f(r) := \dfrac{1}{2\pi}\int\limits_{0}^{2\pi} \text{log}\Vert f(re^{i\theta}) \Vert d\theta -
		\dfrac{1}{2\pi}\int\limits_{0}^{2\pi} \text{log}\Vert f(se^{i\theta}) \Vert d\theta ,( 0< s < r).
	\end{align*}
	\begin{align*}
		N_{f, H}(r) := \int\limits_{s}^{r} \sum_{s\leq |z|\leq t} \min(\nu_{f(H)}(z), n)\dfrac{1}{t}dt
	\end{align*}
	\begin{equation*}
		\delta_f(H)^{[n]}:=1-\lim\sup_{r\rightarrow \infty}\dfrac{N_{f, H}(r)}{T_f(r)}
	\end{equation*}
	As Proposition 4.7 in \cite[page 672]{Fu83}, by using Jensen's formula, we can show that	
	\begin{equation} \label{eq1}
	\delta^H_{f,M}(H)\leq \delta^S_{f,M}(H)\leq \delta_f(H)^{[n]} \leq 1.
	\end{equation} 
The classical defect relation in value distribution theory of meromorphic functions is stated as following.
\begin{lemma}[{\cite[Corollary 3.3.9]{Fu93}}] \label{L7}
	Let $f:\Delta_{s,\infty}\to\mathbb P^n(\mathbb C)$ be a nonconstant holomorphic map where $s$ is a positive real number and let $H_1,\dots,H_q$ be distinct $q$ hyperplanes in $N-$subgeneral position. Assume that $f$ has an essential singularity at $\infty$, then
	$$\sum_{j=1}^{q}\delta_f(H_j)^{[n]}\leq 2N-n+1.$$
\end{lemma}
Using (\ref{eq1}) and Lemma \ref{L7} we have:
\begin{lemma} \label{L8}
	Let $f:\Delta_{s,\infty}\to\mathbb P^n(\mathbb C)$ be a nonconstant holomorphic map where $s$ is a positive real number and let $H_1,\dots,H_q$ be distinct $q$ hyperplanes in $N-$subgeneral position. Assume that $f$ has an essential singularity at $\infty$, then
	$$\sum_{j=1}^{q}\delta_{f,\Delta_{s,\infty}}^H(H_j)\leq 2N-n+1.$$
\end{lemma}
\section{The proof of Main theorem 1}\label{s4}
%\begin{proof}
\indent For the convenience of the reader, we first recall some notations on the Gauss map of minimal surfaces in $\mathbb R^m.$
Let $M$ be a complete immersed minimal surface in $\mathbb R^m.$ Take an immersion $x
=(x_0,\cdots,x_{m-1}) : M \rightarrow \mathbb R^m.$ Then $M$ has the structure of a Riemann surface and any local isothermal coordinate $(\xi_1, \xi_2)$ of $M$ gives a local holomorphic coordinate $z=\xi_1+ \sqrt{-1}\xi_2$. The generalized Gauss map of $x$ is defined to be 
\begin{equation*}
g : M \rightarrow \mathbb P^{m-1}(\mathbb C), g = \mathbb P(\dfrac{\partial x}{\partial z})=(\dfrac{\partial x_0}{\partial z}:\cdots:\dfrac{\partial x_{m-1}}{\partial z}).
\end{equation*}
Since $x:M \rightarrow \mathbb R^m$ is immersed, $$G=G_{z}:= 
(g_0,\cdots,g_{m-1}) = ((g_0)_{z},\cdots,(g_{m-1})_{z}) =(\dfrac{\partial x_0}{\partial z},\cdots,\dfrac{\partial x_{m-1}}{\partial z})$$ is a (local) reduced representation of $g$, and
since for another local holomorphic coordinate $\xi$ on $M$ we have $G_{\xi} = G_{z}\cdot (\dfrac{dz}{d\xi})$, $g$ is well defined (independently of the (local) holomorphic coordinate). 
Moreover, if $ds^2$ is the metric on $M$ induced by the standard metric on $\mathbb R^m$, we have 
\begin{equation}\label{eq:1}
ds^2 = 2|G_{z}|^2|dz|^2 .
\end{equation}
Finally since $M$ is minimal,  $g$ is a holomorphic map. 

Since by hypothesis of Main theorem 1, $g$ is $k$-non-degenerate $(1 \leq k \leq m-1).$
 Without loss of
generality, we may assume that $g(M) \subset \mathbb P^k(\mathbb C);$ then
\begin{equation*}
g : M \rightarrow \mathbb P^{k}(\mathbb C), g = \mathbb P(\dfrac{\partial x}{\partial z})=(\dfrac{\partial x_0}{\partial z}:\cdots:\dfrac{\partial x_{k}}{\partial z})
\end{equation*}
is linearly non-degenerate in $\mathbb P^k(\mathbb C)$ (so in particular $g$ is not constant)  and the other facts mentioned above still hold.

Let $H_j(j = 1,\cdots,q)$ be $q (\geq N+1)$ hyperplanes in $\mathbb P^{m-1}(\mathbb C)$ in $N$-subgeneral position $(N\geq m-1 \geq k)$. Then  $H_j\cap \mathbb P^{k}(\mathbb C)(j = 1,\cdots,q)$ are $q$ hyperplanes in $\mathbb P^{k}(\mathbb C)$ in $N$-subgeneral position. Let each $H_j\cap \mathbb P^{k}(\mathbb C)$ be represented  as
\begin{equation*}
H_j\cap \mathbb P^{k}(\mathbb C) :\overline{c}_{j0}\omega_0+\cdots+\overline{c}_{j{k}}\omega_{k}=0
\end{equation*}
with $\sum_{i=0}^k|\overline{c}_{ji}|^2 = 1.$\\
Set
\begin{equation*}
G(H_j) = G_{z}(H_j):= \overline{c}_{j0}g_0+\cdots+\overline{c}_{jk}g_{k}.
\end{equation*}

We will now, for each contact function $\phi_p(H_j)$ of $g$ for each a hyperplane $H_j$, choose one of the components of the numerator $|((G_z)_p)_z(H_j)|$ which is not identically zero: More precisely, for each $j,p\ (1\leq j \leq q, 1\leq p \leq k),$ we can choose $i_1,...,i_p$ with $0\leq i_1<\cdots< i_p \leq k$ such that
\begin{equation*}
\psi(G)_{jp}=(\psi (G_z)_{jp})_{z}:=\sum_{l\neq i_1,..,i_p}\overline{c}_{jl}W_{z}(g_l,g_{i_1},\cdots,g_{i_p})\not\equiv 0,
\end{equation*}
(indeed, otherwise, we have 
 $\sum_{l\neq i_1,..,i_p}\overline{c}_{jl}W(g_l,g_{i_1},\cdots,g_{i_p})\equiv 0$ for all $i_1, ..., i_p$,
so \\
 $W(\sum_{l\neq i_1,..,i_p}\overline{c}_{jl}g_l,g_{i_1},\cdots,g_{i_p})\equiv 0$ for all $i_1, ..., i_p$, which contradicts the non-degeneracy of $g$ in $\mathbb P^k(\mathbb C).$
Alternatively we simply can observe that in our situation none of the contact functions vanishes identically). We still set $\psi (G)_{j0}=\psi (G_z)_{j0}:=G(H_j)(\not\equiv 0)$, and we also note that $\psi (G)_{jk}=((G_z)_k)_z$. Since the $\psi (G)_{jp}$ are holomorphic, so they have only isolated zeros. 

Finally we put for later use the transformation formulas for all the terms defined above,
which are obtained by using Proposition \ref{W} :
 For local holomorphic coordinates $z$ and $\xi$ on $M$ we have : 
 \begin{equation}\label{e1}
 G_{\xi} = G_{z}\cdot (\dfrac{dz}{d\xi})\,,
 \end{equation}
  \begin{equation}\label{e2}
 G_{\xi}(H) = G_{z}(H) \cdot (\dfrac{dz}{d\xi})\,,
 \end{equation}
   \begin{equation}\label{e3}
  ((G_{\xi})_k)_{\xi}=((G_z)_k)_{z}\cdot (\dfrac{dz}{d\xi})^{k+1+\frac{k(k+1)}{2}}=((G_z)_k)_{z}(\dfrac{dz}{d\xi})^{\sigma_{k+1}}\,,
  \end{equation}
   \begin{equation}\label{e4}
  (\psi (G_{\xi})_{jp})_{\xi}=(\psi (G_z)_{jp})_{z}\cdot (\dfrac{dz}{d\xi})^{p+1+\frac{p(p+1)}{2}}=(\psi (G_z)_{jp})_{z} \cdot (\dfrac{dz}{d\xi})^{\sigma_{p+1}}\,, \:(0 \leq p \leq k)\,.
  \end{equation} 
Moreover, we also will need the following transformation formulas for mixed variables :
   \begin{equation}\label{e5}
  ((G_{\xi})_k)_{\xi}=((G_{\xi})_k)_{z}\cdot (\dfrac{dz}{d\xi})^{\frac{k(k+1)}{2}}=((G_{\xi})_k)_{z}(\dfrac{dz}{d\xi})^{\sigma_{k}}\,,
  \end{equation}
   \begin{equation}\label{e6}
  (\psi (G_{\xi})_{jp})_{\xi}=(\psi (G_{\xi})_{jp})_{z}\cdot (\dfrac{dz}{d\xi})^{\frac{p(p+1)}{2}}=(\psi (G_{\xi})_{jp})_{z} \cdot (\dfrac{dz}{d\xi})^{\sigma_{p}}\,, \:(0 \leq p \leq k)\,.
  \end{equation} 

Now, it follows from hypothesis of Main theorem 1 that
\begin{equation}\label{eq:2}
	\sum_{j=1}^q\delta_{g, A}^H(H_j) > (k+1)(N-\dfrac{k}{2})+(N+1).  
\end{equation} 
By definition, there exist real numbers $\eta_j \geq 0 (1\leq j \leq q)$ such that 
$$q- \sum_{j=1}^q\eta_j > (k+1)(N-\dfrac{k}{2})+(N+1)$$ 
and continuous subharmonic functions $u_j (1\leq j \leq q)$ on $A,$ which are harmonic on $A - g^{-1}(H_j),$ satisfying conditions (C1) and (C2). Thus
\begin{equation}\label{eq:2'}
	\sum_{j=1}^{q}(1- \eta_j)-2N+k-1 > \dfrac{(2N-k +1)k}{2}>0,
\end{equation}
and this implies in particular
\begin{equation} \label{ass1''}
	q>2N-k+1 \geq N+1 \geq k+1 .
\end{equation}
 Since the theorem \ref{N1}, we have
\begin{align}\label{ass1''}
(q - 2N + k - 1)\theta = \sum_{j=1}^q \omega (j) - k - 1\,; \theta \geq \omega(j) > 0 \text{ and }  \theta \geq \dfrac{k + 1}{2N - k + 1}.
\end{align}
So, using (\ref{ass1''}), we get
\begin{align*}
2\bigg ((\sum_{j=1}^q \omega(j)(1-\eta_j)) - k - 1\bigg )&= \dfrac{2((\sum_{j=1}^{q}\omega(j))- k - 1)\theta}{\theta} - 2\sum_{j=1}^{q}\dfrac{\omega(j)\theta \eta_j}{\theta}\\
&= 2(q-2N + k - 1)\theta - 2\sum_{j=1}^{q}\dfrac{\omega(j)\theta \eta_j}{\theta}\\
&\geq 2(q-2N + k - 1)\theta - 2\sum_{j=1}^{q}\theta\eta_j\\
&= 2\theta\bigg (\sum_{j=1}^{q}(1- \eta_j)-2N+k-1\bigg )\\
&\geq 2\dfrac{(k + 1)\bigg (\sum_{j=1}^{q}(1- \eta_j)-2N+k-1\bigg )}{2N-k+1}.
\end{align*}
Thus, we now can  conclude with (\ref{eq:2'}) that 
\begin{align}\label{eq:3}
&2\bigg (\sum_{j=1}^q \omega (j)(1-\eta_j) - k - 1\bigg ) > k(k+1) \nonumber \\
&\Rightarrow \sum_{j=1}^q \omega (j)(1-\eta_j) - k - 1 - \dfrac{k(k+1)}{2} >0.
\end{align}
By (\ref{eq:3}), we can choose a number $\epsilon(>0) \in \mathbb Q $ such that 
\begin{equation*}
	\dfrac{\sum_{j=1}^q \omega(j)(1-\eta_j) - (k + 1) - \frac{k(k + 1)}{2}}{\tau_{k+1}} >  \epsilon > \dfrac{\sum_{j=1}^q \omega (j)(1-\eta_j) - (k + 1) - \frac{k(k + 1)}{2}}{\frac{1}{q} + \tau_{k+1} }.
\end{equation*}
So
\begin{equation} \label{eq:5}
h:=\sum_{j=1}^q \omega(j)(1-\eta_j) - (k + 1) - \epsilon\sigma_{k+1} > \frac{k(k + 1)}{2} + \epsilon\tau_{k}
\end{equation}
and
\begin{equation}\label{eq:6} 
\dfrac{\epsilon}{q}> \sum_{j=1}^q \omega(j)(1-\eta_j) - (k + 1) - \frac{k(k + 1)}{2} - \epsilon\tau_{k+1}.
\end{equation}
We now consider the number
\begin{equation}\label{nb}
\rho := \dfrac{1}{h} \bigg ( \frac{k(k + 1)}{2} + \epsilon\tau_k\bigg)= \dfrac{1}{h} \bigg ( \sigma_k + \epsilon\tau_k\bigg).
\end{equation}
Then, by (\ref{eq:5}), we have
 \begin{equation}\label{eq:7}
0 < \rho < 1.
\end{equation}
Set
\begin{equation}\label{eq:8}
\rho^* := \dfrac{1}{(1 - \rho)h}=\dfrac{1}{\sum_{j=1}^q \omega(j)(1-\eta_j) - (k + 1) - \frac{k(k + 1)}{2} - \epsilon\tau_{k+1}}.
\end{equation}
Using (\ref{eq:6}) we get 
\begin{equation}\label{eq:9}
\dfrac{\epsilon\rho^*}{q} > 1.
\end{equation}

Now, we put 
$$A_1=\{z\in A: \psi (G)_{jp}(z)\not=0 \text{ for all } j=1,\cdots, q \text{ and } p=0,\cdots, k\}.$$
We define a new pseudo metric
\begin{equation}\label{eq:10}
d\tau^2 
= \bigg(\dfrac{\Pi_{j=1}^q|G_{z}(H_j)|^{\omega(j)}}{|((G_z)_k)_{z}|^{1+\epsilon}e^{\sum_{j=1}^{q}\omega (j) u_j}\Pi_{j=1}^q\Pi_{p=0}^{k-1}|(\psi(G_z)_{jp})_{z}|^{\epsilon/q}}\bigg)^{2\rho^*}|dz|^2
\end{equation}
on $A_1.$
We note that by the transformation formulas (\ref{e1}) to (\ref{e4}) 
for a local holomorphic coordinate $\xi$ we have 
\begin{equation}\label{ind1}
\bigg(\dfrac{\Pi_{j=1}^q|G_{z}(H_j)|^{\omega(j)}}{|((G_z)_k)_{z}|^{1+\epsilon}e^{\sum_{j=1}^{q}\omega(j) u_j}\Pi_{j=1}^q\Pi_{p=0}^{k-1}|(\psi (G_z)_{jp})_{z}|^{\epsilon/q}}\bigg)^{2\rho^*}|dz|^2 
\end{equation}
\begin{equation*}
= 
\bigg(\dfrac{\Pi_{j=1}^q|G_{\xi}(H_j)|^{\omega(j)}}{|((G_{\xi})_k)_{\xi}|^{1+\epsilon}e^{\sum_{j=1}^{q}\omega (j) u_j}\Pi_{j=1}^q\Pi_{p=0}^{k-1}|(\psi(G_{\xi})_{jp})_{\xi}|^{\epsilon/q}}\bigg)^{2\rho^*}|d\xi|^2
\end{equation*}
so the pseudo metric $d\tau$ is in fact defined independently of the choice of the coordinate.  

{\it Claim 1: $d\tau$ is continuous and nowhere vanishing on $A_1.$}\\
Indeed, for $z_0 \in A_1$ with $\Pi_{j=1}^qG(H_j)(z_0) \not= 0,$ $d\tau$ is continuous and not vanishing at $z_0.$
Now assume that there exists $z_0 \in A_1$ such that $G(H_i)(z_0)=0$ for some $i.$ Consider the function 
$$ \Gamma (z) =\dfrac{|((G_z)_k)_{z}|^{1+\epsilon}e^{\sum_{j=1}^{q}\omega (j) u_j}\Pi_{j=1}^q\Pi_{p=0}^{k-1}|(\psi(G_z)_{jp})_{z}|^{\epsilon/q}}{\Pi_{j=1}^q|G_{z}(H_j)|^{\omega(j)}} .$$
Combining this with Proposition \ref{P7}, we obtain
\begin{align*}
	\nu_{\Gamma}(z_0)& \geq  \nu_{G_k}(z_0)-\sum_{j=1}^q\omega(j)\nu_{G(H_j)}(z_0)+\sum_{j=1}^q\omega(j)\min\{\nu_{G(H_j)}(z_0), k \}\\
	& +\sum_{j=1}^q\omega(j)\nu_{e^{u_j(z)}|z-z_0|^{-\min(\nu_{G(H_j)}(z_0), k)}}(z_0)  \\ 
	&\geq 0.
\end{align*}
This contradicts to $z_0 \in A_1$. Claim 1 is proved.\\

It is easy to see that $d\tau$ is flat on $A_1$ by definition of $u_j(1\leq j \leq q).$ So, it can be smoothly extended over $K$. Thus, we have a metric, still call it $d\tau,$ on $$A'_1=A_1\cup K$$ 
that is flat outside the compact set $K$. 

{\it Claim 2: $d\tau$ is complete on $A'_1.$ }\\
 We proceed by contradition. If $A'_1$ isn't complete, there is a divergent curve $\gamma(t)$ on $A'_1$ with finite length. We may assume that there is a positive distance $d$ between curve $\gamma$ and the compact $K$. Therefore $\gamma: [0,1)\to A_1$ and $\gamma$ divergent on $A'_1$, with finite length. It implies that from the point of view of $M$, there are two caces: either $\gamma(t)$ tends to a point $z_0$ with 
$$\Pi_{p=0}^{k}\Pi_{j=1}^q|\psi (G)_{jp}|(z_0) = 0.$$
($\gamma(t)$ tends to the boundary of $A'_1$ as $t\to 1$) or else $\gamma(t)$ tends to the boundary of $M$ as $t\to 1.$ 

For the former case, by (\ref{eq:10}) and Proposition \ref{P7} we have 
\begin{align*}
\nu_{d\tau}(z_0)&= - \bigg(  \nu_{G_k}(z_0)-\sum_{j=1}^q\omega(j)\nu_{G(H_j)}(z_0)+\sum_{j=1}^q\omega(j)\min\{\nu_{G(H_j)}(z_0), k \}\\
& +\sum_{j=1}^q\omega(j)\nu_{e^{u_j(z)}|z-z_0|^{-\min(\nu_{G(H_j)}(z_0), k)}}(z_0)+ (\epsilon \nu_{G_k}(z_0) + \dfrac{\epsilon}{q}\sum_{j=1}^q \sum_{p=0}^{k-1}\nu_{\psi (G)_{jp}}(z_0))\bigg)\rho^*\\ 
&\leq -\epsilon \rho^*\nu_{G_k}(z_0) -\dfrac{\epsilon\rho^*}{q}\sum_{j=1}^q \sum_{p=0}^{k-1}\nu_{\psi (G)_{jp}}(z_0)\leq -\dfrac{\epsilon\rho^*}{q}.
\end{align*}
Thus we can find a positive constant $C$ such that 
\begin{equation*}
|d\tau| \geq \dfrac{C}{ |z-z_0|^{\frac{\epsilon\rho^*}{q}}}|dz|
\end{equation*}
in a neighborhood of $z_0$ and then, combining with (\ref{eq:9}), we thus have $$\int_0^1d\tau =\infty$$ contradicting the finite length of $\gamma$. Therefore the last case occur, that is $\gamma(t)$ tends to the boundary of $M$ as $t\to 1.$ 

 Choose $t_0$ such that $$\int_{t_0}^1d\tau <{d}/3.$$ We consider a small disk $\Delta$ with center at $\gamma (t_0).$ Since $d\tau$ is flat, by Lemma \ref{L5}, $\Delta$ is isometric to an ordinary disk in the plane. Let $\Phi: \{|w| < \eta \}\rightarrow \Delta$ be this isometry. Extend $\Phi,$ as a local isometry into $A_1,$ to the largest disk $\{|w| < R\} = \Delta_R$ possible. Then $R \leq {d}/3.$ Hence, the image under $\Phi$ be bounded away from $K$ by distance at least ${2d}/3.$ The reason that $\Phi$ cannot be extended to a larger disk is that the image goes to the outside boundary $A'_1$ (it cannot go to points of $A'_1$ with $\Pi_{p=0}^{k}\Pi_{j=1}^q|\psi (G)_{jp}|(z_0) = 0$ since we have shown already to be infinitely far away in the metric with respect to these points). 
More precisely, by again Lemma \ref{L5}, there exists a point $w_0$ with $|w_0| =R$ so that $\Phi(\overline{0,w_0}) = \Gamma_0$ is a divergent curve on $M.$ 

We now want to use Lemma \ref{ML} to finish up Claim 2 by showing that $\Gamma_0$ has finite length in the original $ds^2$ on $M$, contradicting the completeness of the $M$. For the rest of the proof of Claim 2
we consider $G_z=((g_0)_z,...,(g_k)_z)$ as a {\it fixed globally defined reduced representation of $g$} by means of the global coordinate $z$ of $A \supset A_1$.
(We remark that then we loose of course the invariance of $d\tau^2$ under coordinate changes (\ref{ind1}), but since $z$ is a global coordinate this will be no problem 
and we will not need this invariance
for the application of Lemma \ref{ML}.) If again 
$\Phi : \{w: |w| < R\} \rightarrow A_1$ is our maximal local isometry, it is in particular holomorphic and locally biholomorphic. So $f:= g \circ \Phi :  \{w: |w| < R\} \rightarrow
\mathbb P^k (\mathbb C)$ is a linearly non-degenerate holomorphic map with fixed global
reduced representation $$F:= G_z \circ \Phi =((g_0)_z  \circ \Phi,\cdots,(g_k)_z  \circ \Phi)
=(f_0,\cdots,f_k)\,.$$ 
Since $\Phi$ is locally biholomorphic, the metric on $\Delta_R$ induced from $ds^2$ (cf. (\ref{eq:1})) through $\Phi$ is given by
\begin{equation}\label{eq:12}
\Phi^*ds^2 =  2|G_z \circ \Phi|^2|\Phi^* dz|^2 = 
2|F|^2 |\frac{dz}{dw}|^2|dw|^2\,.
\end{equation}
On the other hand, $\Phi$ is locally isometric, so we have
\begin{equation*}
|dw| = |\Phi^*d\tau|= \bigg(\dfrac{\Pi_{j=1}^q|G_z(H_j) \circ \Phi|^{\omega(j)}}{|((G_z)_k)_z \circ \Phi|^{1+\epsilon}e^{\sum_{j=1}^{q}\omega (j) {u_j}\circ\Phi}\Pi_{p=0}^{k-1}\Pi_{j=1}^q|(\psi (G_z)_{jp})_z \circ \Phi|^{\epsilon/q}}\bigg)^{\rho^*}|\frac{dz}{dw}||dw|\,.
\end{equation*}
By (\ref{e5}) and (\ref{e6}) we have
$$
 ((G_z)_k)_z \circ \Phi =((G_z \circ \Phi)_k)_{w}(\dfrac{dw}{dz})^{\sigma_{k}}
 =(F_k)_{w}(\dfrac{dw}{dz})^{\sigma_{k}}\,,$$
 $$
  (\psi (G_{z})_{jp})_{z} \circ \Phi =(\psi (G_{z} \circ \Phi)_{jp})_{w} \cdot (\dfrac{dw}{dz})^{\sigma_{p}}=(\psi (F)_{jp})_{w} \cdot (\dfrac{dw}{dz})^{\sigma_{p}}\,, \:(0 \leq p \leq k)\,\,.
$$
Hence, by definition of $\rho$ in (\ref{nb}), we have
\begin{align*}
|\dfrac{dw}{dz}|&= 
\bigg(\dfrac{\Pi_{j=1}^q|G_z(H_j) \circ \Phi|^{\omega(j)}}{|((G_z)_k)_z \circ \Phi|^{1+\epsilon}e^{\sum_{j=1}^{q}\omega (j) {u_j}\circ\Phi}\Pi_{p=0}^{k-1}\Pi_{j=1}^q|(\psi (G_z)_{jp})_z \circ \Phi|^{\epsilon/q}}\bigg)^{\rho^*}
\\
&= \bigg(\dfrac{\Pi_{j=1}^q|F(H_j)|^{\omega(j)}}{|(F_k)_w|^{1+\epsilon}e^{\sum_{j=1}^{q}\omega (j) {u_j}\circ\Phi}\Pi_{p=0}^{k-1}\Pi_{j=1}^q|(\psi (F)_{jp})_w|^{\epsilon/q}}\bigg)^{\rho^*}\dfrac{1}{|\frac{dw}{dz}|^{h\rho\rho^*}}.
\end{align*}
So by the definition of $\rho^*$ in (\ref{eq:8}), we get
\begin{align*}
|\dfrac{dz}{dw}|&= 
\bigg(
\dfrac{|(F_k)_w|^{1+\epsilon}e^{\sum_{j=1}^{q}\omega (j) {u_j}\circ\Phi}\Pi_{p=0}^{k-1}\Pi_{j=1}^q|(\psi (F)_{jp})_w|^{\epsilon/q}}
{\Pi_{j=1}^q|F(H_j)|^{\omega(j)}}
\bigg)^{\frac{\rho^*}{1+h\rho\rho^*}}\\
&= 
\bigg(
\dfrac{|(F_k)_w|^{1+\epsilon}e^{\sum_{j=1}^{q}\omega (j) {u_j}\circ\Phi}\Pi_{p=0}^{k-1}\Pi_{j=1}^q|(\psi (F)_{jp})_w|^{\epsilon/q}}
{\Pi_{j=1}^q|F(H_j)|^{\omega(j)}}
\bigg)^{\frac{1}{h}} \ .
\end{align*}
Moreover, $|(\psi (F)_{jp})_w| \leq |(F_p)_w(H_j)|$ by the definitions, so we obtain
\begin{equation}\label{eq:13}
|\dfrac{dz}{dw}| \leq  
\bigg(\dfrac{|(F_k)_w|^{1+\epsilon}e^{\sum_{j=1}^{q}\omega (j) {u_j}\circ\Phi}\Pi_{p=0}^{k-1}\Pi_{j=1}^q|(F_p)_w(H_j)|^{\epsilon/q}}
{\Pi_{j=1}^q|F(H_j)|^{\omega(j)}}\bigg)^{\frac{1}{h}} \ .
\end{equation}
By (\ref{eq:12}) and (\ref{eq:13}), we have
\begin{equation*}
\Phi^*ds \leqslant \sqrt{2}|F|
\bigg(\dfrac{|(F_k)_w|^{1+\epsilon}e^{\sum_{j=1}^{q}\omega (j) {u_j}\circ\Phi}\Pi_{p=0}^{k-1}\Pi_{j=1}^q|(F_p)_w(H_j)|^{\epsilon/q}}
{\Pi_{j=1}^q|F(H_j)|^{\omega(j)}}\bigg)^{\frac{1}{h}}
|dw|.
\end{equation*}
Since (\ref{eq:5}) all the conditions of Lemma \ref{ML} are satisfied. So we obtain the following from Lemma \ref{ML} :
\begin{equation*}
\Phi^*ds \leqslant C(\dfrac{2R}{R^2 -|w|^2})^{\rho}|dw|\,
\end{equation*}
for some constant $C.$ It follows from (\ref{eq:7}) that $0 < \rho < 1.$ Then  
\begin{equation*}
d_{\Gamma_0} \leqslant \int_{\Gamma_0}ds = \int_{\overline{0,w_0}}\Phi^*ds \leqslant C \cdot \int_0^R(\dfrac{2R}{R^2 -|w|^2})^{\rho}|dw|  < + \infty,
\end{equation*}
where $d_{\Gamma_0}$ denotes the length of the divergent curve $\Gamma_0$ in $M,$ contradicting the assumption of completeness of $M.$  Thus, we complete Claim 2. 

Since the metric on $A'_1$ is flat outside of a compact set $K$ and Claim 2, by a theorem of Huber (\cite{Hub}, Theorem 13, p.61) the fact that $A'_1$ has finite total curvature implies that $A'_1$ is finitely connected. This means that there is a compact subregion of $A'_1$ whose complement is the union of a finite number of doubly-connected regions. Thus, we can first conclude that $\Pi_{p=0}^{k}\Pi_{j=1}^q|\psi (G)_{jp}|(z)$ can have only a finite number of zeros, and second, that the original surface $M$ is finitely connected. Furthermore, by Osserman (\cite{O63}, Theorem 2.1) each annular ends of $A'_1$, hence of $M,$ is conformally equivalent to a punctured disk. Thus, the Riemann surface $M$ must be conformally equivalent to a compact Riemann surface $\overline{M}$ with a finite number of points removed. In a neighborhood of each of those points the Gauss map $G$ satisfies 
$$\sum_{j=1}^q\delta_G^H(H_j)> (k+1)(N-\dfrac{k}{2})+(N+1)>2N-k+1.$$ 
Now by Lemma \ref{L8}, the Gauss map $G$ is not essential at those points. Therefore $G$ can be extended to a holomorphic map from $\overline{M}$ to $\mathbb P^k(\mathbb C)$. If the homology class represented by the image of $G:\overline{M}\to\mathbb P^k(\mathbb C)$ is $l$ times the fundamental homology class of $\mathbb P^k(\mathbb C),$ then we have $$\iint KdA=-2\pi l$$ as the total curvature of $M$. This proves Main theorem 1. 
%\end{proof}
%%%%%%%%%%%%%%%%%%%%%%%%%%%%%%%%%%%%%%%%%%%%%%%%%%%%%%%%%%%%%%%%%
\section{The proof of  Main theorem 2}\label{s5}
%\begin{proof}
\indent For convenience of the reader, we first recall some notations on the Gauss map of minimal surfaces in $\mathbb R^3.$
\indent Let $x =( x_1, x_2, x_3) : M \rightarrow \mathbb R^3$ be a non-flat complete minimal surface and $g: M \rightarrow  \mathbb P^1(\mathbb C)$ its Gauss map.  Let $z$ be a local holomorphic coordinate. Set $\phi_i := \partial x_i / \partial z \ (i = 1, 2, 3 )$ and $\widehat{\phi}:= \phi_1-\sqrt{-1}\phi_2.$ Then, the (classical) Gauss map $g: M \rightarrow \mathbb P^1(\mathbb C)$ is given by  $$g=\dfrac{\phi_3}{\phi_1 - \sqrt{-1}\phi_2},$$
and the metric on $M$ induced from $\mathbb R^3$ is given by
$$ds^2= |\widehat{\phi}|^2(1 + |g|^2)^2|dz|^2  \text{ (see \cite{Fu93})}.$$
We remark that although the $\phi_i$, $(i=1,2,3)$ and $\widehat{\phi}$ depend on $z$, $g$ and $ds^2$ do not.
Next we take a reduced representation $g = (g_0 : g_1)$ on $M$ and set $||g|| = (|g_0|^2 +|g_1|^2)^{1/2}.$ Then we can rewrite 
\begin{equation}
ds^2 = |h|^2||g||^4|dz|^2\,, 
\end{equation}
where $h:= \widehat{\phi}/g_0^2.$ In particular, $h$ is a holomorphic map without zeros. We remark that  $h$ depends on $z$, however, the reduced representation $g=(g_0:g_1)$ is globally defined on $M$ and independent of $z$.
Finally we observe that by the assumption that $M$ is not flat,  $g$ is not constant.

Now the proof of Main theorem 2 will be completely analogue to the proof of Main theorem 1.

Firstly, for each $a^j\:( 1\leq j \leq q )$  be distinct points in $\mathbb P^1(\mathbb C)$, we may assume $a^j=(a^j_0: a^j_1 )$ with $|a^j_0|^2 + |a^j_1|^2 = 1$ $( 1\leq j\leq q )$. We set
$G_{j}:=a^j_0g_1 - a^j_1g_0 \ (1\leq j \leq q)$ for the reduced representation $g = (g_0 : g_1)$ of the Gauss map. 
 
Secondly, it follows from hypothesis of  Main theorem 2 that
\begin{equation}\label{eq:2}
\sum_{j=1}^q\delta_{g, A}^H(a_j) > 4.  
\end{equation} 
By definition, there exist real numbers $\eta_j \geq 0 (1\leq j \leq q)$ such that 
$$q- \sum_{j=1}^q\eta_j > 4$$ 
and continuous subharmonic functions $u_j (1\leq j \leq q)$ on $A,$ which are harmonic on $A - g^{-1}(a_j),$ satisfying conditions (C1) and (C2). Thus we can take $\delta$ with
$$ \dfrac{q -4-\sum_{j=1}^{q}\eta_j}{q} > \delta > \dfrac{q -4-\sum_{j=1}^{q}\eta_j}{q +2},$$
and set $p = 2/ (q -2-\sum_{j=1}^{q}\eta_j-q\delta).$ Then 
\begin{equation} \label{3.4.1}0 < p < 1 , \ \frac{p}{1-p} > \frac{\delta p}{1-p} > 1 \ . 
\end{equation} 
For convenience, we will use again some notations as in the proof of Main theorem 1.\\
Put $A_1=\{z\in A: W(g_0,g_1)(z)\not=0 \text{ for all } j=1,\cdots, q\}.$\\
 We define a new pseudo metric
$$d\tau^2 = |h|^{\frac{2}{1-p}}\bigg(\dfrac{\Pi_{j=1}^{q}|G_{j}|^{1-\delta}}{|W(g_0,g_1)|e^{\sum_{j=1}^qu_j}}\bigg)^{\frac{2p}{1-p}}|dz|^2 \ $$
on $A_1$ (where again $G_{j} := a^j_0g_1 - a^j_1g_0$ and $h$ is defined with respect
to the coordinate $z$ on $A_1$ and $W(g_0,g_1) = W_z(g_0,g_1)$).

 For last steps, we argue similarly to the proof of Main theorem 1. Thus  we finished the proof of Main theorem 2.
%\end{proof}

\vspace{1cm}

\end{document}